\documentclass[12pt]{article}
\usepackage{amsmath,graphicx}
\usepackage{e-jc}
\specs{P1.18}{28(1)}{2021}{9550}

\usepackage{amssymb}
\usepackage{mathtools}
\usepackage{enumerate}

\usepackage{tikz}
\usepackage{color}
\usepackage{comment}

\theoremstyle{definition}

\newtheorem{conj}[theorem]{Conjecture}
\newtheorem{cor}[theorem]{Corollary}
\newtheorem{prop}[theorem]{Proposition}

\newtheorem*{main}{Main Theorem}

\newtheoremstyle{casestyle}
{3pt}
{3pt}
{}
{}
{\bfseries}
{:}
{.5em}
{}

\theoremstyle{casestyle}
\newtheorem{case}{Case}[theorem]

\newcommand{\cC}{\mathcal{C}}
\newcommand{\cG}{\mathcal{G}}

\newcommand{\Ind}{{\mathcal I}}

\newcommand{\N}{\mathbb{N}}

\newcommand{\of}{\subseteq}

\renewcommand{\k}{\widetilde{k}}

\newcommand{\wo}{\setminus}

\newcommand{\card}{\#}

\newcommand{\0}{\emptyset}

\renewcommand{\complement}{\overline}
\renewcommand{\d}{\delta}

\newcommand{\fl}{\phi}
\newcommand{\Gp}[1]{G_{#1}}

\DeclarePairedDelimiter\abs{|}{|}
\DeclarePairedDelimiter\parens{(}{)}
\DeclarePairedDelimiter\set{\{}{\}}

\DeclarePairedDelimiterX\setof[2]{\{}{\}}{#1\,:\,#2}

\DeclareMathOperator\CT{CT}
\DeclareMathOperator\ex{ex}
\DeclareMathOperator\mex{mex}

\dateline{May 20, 2020}{Jan 3, 2021}{Jan 29, 2021}


\Copyright{The authors. Released under the CC BY-ND license (International 4.0).}

\title{Many cliques with few edges}

\author{Rachel Kirsch\thanks{Supported in part by NSF grant DMS-1839918.}\\
\small Department of Mathematics\\[-0.8ex]
\small Iowa State University\\[-0.8ex] 
\small Ames, Iowa, U.S.A.\\
\small\tt kirsch@iastate.edu\\
\and
A.~J. Radcliffe\thanks{Supported in part by Simons Foundation grant number 429383.}\\
\small Department of Mathematics\\[-0.8ex]
\small University of Nebraska-Lincoln\\[-0.8ex]
\small Lincoln, Nebraska, U.S.A.\\
\small\tt jamie.radcliffe@unl.edu}

\begin{document}

\maketitle


\begin{abstract}
Recently Cutler and Radcliffe proved that the graph on $n$ vertices with maximum degree at most $r$ having the most cliques is a disjoint union of $\lfloor n/(r+1)\rfloor$ cliques of size $r+1$ together with a clique on the remainder of the vertices. It is very natural also to consider this question when the limiting resource is edges rather than vertices. In this paper we prove that among graphs with $m$ edges and maximum degree at most $r$, the graph that has the most cliques of size at least two is the disjoint union of $\bigl\lfloor m \bigm/\binom{r+1}{2} \bigr\rfloor$ cliques of size $r+1$ together with the colex graph using the remainder of the edges. 
\end{abstract}

\section{Introduction}

There has been a lot of recent work on the general problem of determining which graphs have the most cliques subject to natural ``resource limitations'' and additional constraints. Most of the early work focused on vertices as resources, and added other conditions. There are results about $n$-vertex graphs of maximum degree at most $r$. For instance, Cutler and the second author proved that the unique graph on $n$ vertices with maximum degree at most $r$ having the most cliques is $G = aK_{r+1}\cup K_b$, with $a$ as large as possible. Gan, Loh, and Sudakov \cite{GLS15} made substantial progress on the natural conjecture that the same graph has the largest number of cliques of any fixed size $t\ge 3$. Very recently Chase \cite{C19} proved this conjecture. 

Zykov \cite{Z49} proved that the Tur\'an graph has the largest number of $t$-cliques among all $n$-vertex graphs containing no $K_{r+1}$. There is a large literature on the very closely related problem of determining which $n$-vertex graphs have the largest number of independent sets \cite{EG14,G11}.

In this paper we consider results for which the resource is edges. We fix the number of edges in the graph, possibly impose other conditions, and ask which graph has the largest number of cliques. 

Since we are putting no constraint on the number of vertices, the simplest version of the problem is that of determining
\[
	\max\setof[\big]{\k(G)}{\text{$G$ has $m$ edges}},
\]
where we write $k_t(G)$ for the number of cliques in $G$ of size $t$ and
\[
	\k(G) = \sum_{t\ge 2} k_t(G)
\]
for the number of cliques in $G$ of size at least $2$. This question is straightforward to answer. The Kruskal-Katona Theorem \cite{K63,K68} easily shows that this maximum is achieved when $G=\cC(m)$, the colex graph having $m$ edges. This is the graph on vertex set $\N$ whose edges are the first $m$ in the colexicographic (colex) order on pairs.

\begin{prop}\label{prop:kkt} For $t \ge 2$, if a graph $G$ has $m$ edges, then $k_t(G) \le k_t(\cC(m))$. In particular, if $m = \binom{q}2$, then $k_t(G)\le \binom{q}t$.
\end{prop}

\begin{proof}
Immediate from the Kruskal-Katona Theorem. See \cite{KR19} for details.
\end{proof}

Radcliffe and Uzzell noted that it is a straightforward consequence of the ``rainbow Kruskal-Katona Theorem'' of Frankl, F\"{u}redi, and Kalai \cite{FFK88}, together with an important result of Frohmader \cite{Frohmader08} that 
\[
	\max\setof[\big]{\k(G)}{\text{$G$ has $m$ edges and contains no $K_{r+1}$}},
\]
is achieved by the $r$-partite colex-Tur\'an graph $\CT_r(m)$. For more details see \cite{RU18}.

\subsection{Main Result} 
\label{sub:main_result}

In this paper we consider the next most natural problem in this area. We determine
\[
	f(m, r) = \max\setof{\k(G)}{\text{$G$ has $m$ edges and $\Delta(G)\le r$}}.
\]
Indeed we show that the maximum is attained by a graph made up of as many copies of $K_{r+1}$ as it is possible to build with $m$ edges, together with an additional component (possibly empty) that is a colex graph on strictly fewer than $\binom{r+1}2$ edges. For convenience we give a name to the value of $\k$ of this graph. Defining $a$ and $b$ by $m = a\binom{r+1}2 + b$, $0\le b<\binom{r+1}2$ we let
\[
	g(m, r) = \k(aK_{r+1}\cup \cC(b)) = a(2^{r+1}-r-2)+\k(\cC(b)).
\]
It is possible to describe the last term more carefully. For any $0\le b<\binom{r+1}2$, there exist unique $c$ and $d$ defined by $b = \binom{c}{2} + d$, $0 \le d < c$. Moreover $\cC(b)$ consists of a clique of size $c$, together with another vertex joined to $d$ vertices of the clique. We have
\[
	\k(\cC(b)) = 2^c-c+2^d-2.
\]

Thus our main theorem is as follows.

\begin{main}
	For all $m,r\in \N$, write $m=a\binom{r+1}{2}+b$ with $0\le b < \binom{r+1}2$. If $G$ is a graph on $m$ edges with $\Delta(G)\leq r$, then 
	\begin{equation}
		\k(G) \leq \k(aK_{r+1} \cup \cC(b)),\label{eq:theone}
	\end{equation}
	with equality if and only if (disregarding isolated vertices) $G=aK_{r+1}\cup \cC(b)$ or $G=aK_{r+1}\cup K_c\cup K_2$ (where $b = \binom{c}{2}+1$).
\end{main}

In \cite{KR19}, we conjectured the following refinement of this main theorem.

\begin{conj}\label{conj:allt}For any $t \ge 3$, if $G$ is a graph with $m$ edges and maximum degree at most $r$, then 
\[
	k_t(G) \le k_t(aK_{r+1}\cup \cC(b)),
\]
where $m = a\binom{r+1}{2}+b$ and $0 \le b < \binom{r+1}{2}$.
\end{conj}

We proved this conjecture when $t=3$ and $r \le 8$:
\begin{theorem}[\cite{KR19}]\label{thm:kr19}For any $r \le 8$, if $G$ is a graph with $m$ edges and maximum degree at most $r$, then 
\[
	k_3(G) \le k_3(aK_{r+1}\cup \cC(b)),
\]
where $m = a\binom{r+1}{2}+b$ and $0 \le b < \binom{r+1}{2}$.
\end{theorem}

Our proof of the main theorem uses several ``local moves'': alterations to a potentially optimal graph demonstrating that a closely related graph would have a strictly larger value of $\k$. We handle graphs to which no local move applies by a global averaging argument. 

Our initial analysis of the structure of a potentially optimal graph is to consider \emph{tight cliques}. A clique, of size $t$ say, is \emph{tight} if its vertices have $r+1-t$ common neighbors. A maximal tight clique is called a \emph{cluster}. We expect that optimal graphs will have many tight cliques. Note that it is impossible to have a cluster of size $r$ since by definition such a tight clique would be contained in a $K_{r+1}$ component, i.e. a cluster of size $r+1$. If a cluster of size $r+1$ exists, we will be able to apply induction on $m$. Also, if $G \in \cG(m,r)$ is connected and $m > \binom{r+1}{2}$, then all clusters in $G$ have size at most $r-1$. Much of our proof involves considering the possible structures and relationships of clusters. 

In Section~\ref{sec:clusters} we outline the basic properties of clusters. In Section~\ref{sec:local_moves} we discuss the various local alterations we will attempt to do on a potentially optimal connected graph.
All are variants of an operation we call \emph{folding}, which was introduced in a slightly different context in \cite{CR14} and also used in \cite{CR2016} and \cite{KR19}.
If we can establish that the folded graph has no more edges and no higher maximum degree than before, and $\k$ has strictly increased, then we can eliminate $G$ as a potentially optimal graph. In Section \ref{sec:averaging} we use a global averaging argument to show that the remaining connected graphs are not extremal, and in Section~\ref{sec:proof} we prove our Main Theorem. We finish in Section~\ref{sec:conclusion} by describing some open problems in the area.


\section{Clusters} 
\label{sec:clusters}

Let us begin by establishing some notation. Most of our notation is standard (see, for instance, \cite{BB}). We will write $N[v]$ for the closed neighborhood of a vertex $v$. If $e$ is an edge of $G$ then we let $k_t(e)$ denote the number of $t$-cliques of $G$ containing $e$. Most of the graphs we consider will have $m$ edges and maximum degree at most $r$. We write $\cG(m,r)$ for the set of such graphs and $\cG_C(m,r)$ for the set of connected graphs in $\cG(m,r)$.

One way to think about counting the number of cliques in a graph $G$ containing $m$ edges, is to count how the $3$-cliques are made up of $2$-cliques, how the $4$-cliques are made up of $3$-cliques, and so on. Clearly each $t$-clique contains $t$ copies of $K_{t-1}$, and equally easily each $t$-clique $K$ is contained in $\abs{N(K)}$ copies of $K_{t+1}$, where $N(K)$ is the set of common neighbors of $K$. Thus, for $G$ with maximum degree at most $r$ to have many cliques it seems likely that many of its cliques must have as large a common neighborhood as possible. This prompts the following definitions.

\begin{definition}
	If $G$ is a graph of maximum degree at most $r$ and $T\of V(G)$ is a clique satisfying $\abs{N(T)} = r+1-\abs{T}$, then we say that $T$ is \emph{tight}. If $T$ is a maximal tight clique, then we say it is a \emph{cluster}.
\end{definition}

For a cluster $T$ we can analyze the graph locally as consisting of $T$, its neighbors, and connections to the rest of the graph. The following definition establishes our notation and conventions. 

\begin{definition}
	Suppose that $G \in \cG_C(m,r)$ is a graph with a cluster $T\of V(G)$. We let $S_T = N(T)$ and let $B_T$ be the graph of edges $uv$ such that $u \in S_T$ and $v \in V(G)\setminus(T\cup S_T)$. We define
	\[
		R_T = \complement{G[S_T]},
	\]
i.e., the graph on $S_T$ whose edges are those not in $G$. We think of the missing edges $R_T$ as \emph{red edges}, and the edges $B_T$ from the cluster to the rest of the graph as \emph{blue edges}. See Figure \ref{fig:cluster}. When the risk of confusion is low we will simply refer to $S_T$, $R_T$, and $B_T$ as $S$, $R$, and $B$. Note that since $G$ has maximum degree at most $r$ we know that for all vertices $x\in S$ we have $d_B(x) \le d_R(x)$.
\end{definition}

\begin{figure}[ht]
\begin{center}
    \tikzstyle{vx}=[inner sep=1.5pt,circle,fill=black,draw=black]
    \tikzstyle{R edge}=[red,thick]
    \tikzstyle{B edge}=[blue,thick]
    \tikzstyle{edge}=[thick]
    
	\begin{tikzpicture}
	    \clip (-0.2,-0.2) rectangle (5.7,4.2);
	    \clip (2.5,2) ellipse [x radius=3, y radius=2.5];
        \draw (1,2) ellipse [x radius=0.75, y radius=1.5];
        \node at (1,1) {$T$};
        \draw[step=0.15, thick, rotate around={20:(1,2.8)}] (.9,2.7) grid +(0.449,0.449);  %
        \draw (4,2) ellipse [x radius=1, y radius=2];
        \node at (4,0.5) {$S$};
        \draw[thick] (2,2.25) -- (2.75,2.25) -- (2,1.75) -- (2.75,1.75);
        \foreach \x/\y [count=\i, evaluate=\y as \righty using 6*\y-13] in {3.5/1.5, 4.3/1.8, 4/2.2, 4.6/2.5, 3.5/2.6, 4.1/3.2}
        {
            \node[vx] (V\i) at (\x,\y) {};
            \coordinate (U\i) at (6,\righty);
        }
        \foreach \i/\j in {1/3, 2/3, 2/4, 3/4, 3/5, 3/6, 4/5, 4/6}
        {
            \draw[B edge] (V\i) to[bend right] (U\j);
            \draw[B edge] (V\j) to[bend left] (U\i);
            \draw[R edge] (V\i) -- (V\j);
        }
    \end{tikzpicture}
	\caption{A cluster and its neighborhood}\label{fig:cluster}
\end{center}
\end{figure}
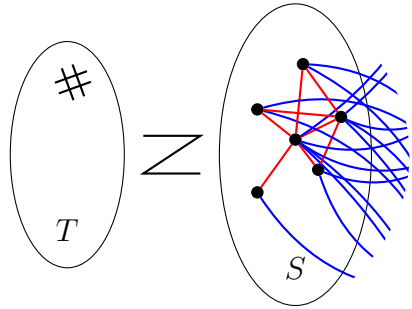

\begin{lemma}\label{lem:clustructure}
	Let $G$ be a graph with maximum degree at most $r$.\begin{enumerate}
	\item Two vertices $x,y\in V(G)$ of degree $r$ belong to some common tight clique precisely if $N[x]=N[y]$. In particular the relation of belonging to some common tight clique is an equivalence relation on vertices of degree $r$. The clusters of $G$ are disjoint and are the equivalence classes.
	\item An edge can be incident to at most two clusters.
	\item An edge $e$ that is not tight (equivalently, that is not in a cluster) has $k_t(e) \le \binom{r-2}{t-2}$ for every $t \ge 2$.
	\end{enumerate}
\end{lemma}

\begin{proof}
	Immediate.
\end{proof}


\section{Local Moves} 
\label{sec:local_moves}

In this section we define three ways to alter a potentially extremal connected graph $G$ when it has a cluster $T$ with certain properties. We call these local moves \emph{folding}, \emph{colex folding}, and \emph{partial folding}. We use colex folding (Subsection \ref{subsec:colex}) and partial folding (Subsection \ref{subsec:partial}) to show that graphs with certain types of clusters are not extremal.

\subsection{Folding}\label{subsec:folding}

As in \cite{CR14}, \cite{CR2016}, and \cite{KR19}, we define the operation of folding as follows.

\begin{definition}
	Suppose that $G \in \cG_C(m,r)$ and $T\of V(G)$ is a cluster. The \emph{folding of $G$ at $T$} is the new graph $G_T$ formed by converting $T\cup S$ into a clique (of size $r+1$) and deleting all the blue edges. In other words we define
	\[
		\Gp{T} = G + \binom{S}{2} - E(B).
	\]
\end{definition}

The folded graph $G_T$ also has maximum degree at most $r$ since each vertex in $T\cup S$ now has degree equal to $r$, and the degree of each other vertex has weakly decreased. Nevertheless, for the clusters $T$ that have $e(B) < e(R)$, we have $e(G_T) > e(G)$, so a comparison to $G_T$ cannot determine whether $G$ is extremal. In contrast, when $e(B) \ge e(R)$, we have $G_T\cup (e(G)-e(G_T))K_2 \in \cG(m,r)$. Thus we consider folding only when $e(B) \ge e(R)$. We define \emph{unfoldable} clusters as those with $\k(G) \ge \k(G_T)$ (whether or not $e(G_T) \le e(G)$). We use unfoldability in Subsection \ref{subsec:FL}.

\subsection{Colex Folding}\label{subsec:colex}

Colex folding is a natural variant of folding to consider when $e(B) < e(R)$.

\begin{definition}
	Suppose that $G\in \cG_C(m,r)$ and $T\of V(G)$ is a cluster. The \emph{colex folding} of $G$ at $T$ is the new graph $G'_T$ formed by deleting $T\cup S$ and adding a colex graph with the appropriate number of edges. In other words we define
	\[
		G'_T = G[V(G)\setminus(T\cup S)] \cup \cC\parens[\Big]{\binom{r+1}{2}-e(R)+e(B)}.
	\]
\end{definition}

By construction $e(G'_T) = e(G)$. When $e(B) < e(R)$, the number of edges in the new colex component is less than $\binom{r+1}{2}$, so no vertex will have degree greater than $r$, and $G'_T \in \cG(m,r)$. We will show that if some cluster $T$ in $G$ has $e(B) < e(R) \le r$, then $\k(G) \le \k(G'_T)$. In going from $G$ to $G'_T$ we lose all cliques that contain a blue edge, but we will show that we get a net gain from cliques we create in the colex component.

We call cliques containing a blue edge \emph{blue cliques}. We start by proving a series of statements bounding the number of blue cliques.

\begin{prop}\label{prop:bluetbound}For any cluster $T$ in a graph $G \in \cG_C(m,r)$ and any $t \ge 2$, the number of blue $K_t$'s is at most $\binom{e(B)}{t-1}$.\end{prop}

\begin{proof}
Given a $t$-clique $C$ that contains a blue edge, its blue edges form a spanning and connected (indeed, complete bipartite) subgraph of $C$, so $C$ has a blue spanning tree. This map from blue $K_t$'s to $(t-1)$-sets of blue edges is an injection.
\end{proof}

\begin{lemma}\label{lem:numt}
For $u, q \in \N$ such that $\binom{q}{2} \le u$ and $t\ge 2$,
\[
	k_t(\cC(u)) + \binom{q}{t-1} \le k_t(\cC(u+q)).
\]
\end{lemma}

\begin{proof}
By Pascal's identity, this inequality is equivalent to $k_t(G_1) \le k_t(G_2)$, where $G_1 = K_{q+1} \cup \cC(u)$ and $G_2 = K_q \cup \cC(u+q)$. We can change $G_1$ to $G_2$ by deleting one vertex of the $K_{q+1}$ and adding its $q$ edges to the colex component.

In $K_{q+1}$ there are $\binom{q}{t-1}$ $K_t$'s that contain a given vertex; these are the $K_t$'s lost. We show that at least $\binom{q}{t-1}$ $K_t$'s are gained by considering two cases based on $c$ and $d$ defined by $u = \binom{c}{2} + d$, $0 \le d < c$.

\begin{case}$q \le c-d$
\end{case}

In this case, the $q$ new edges form a star in the colex component. They contribute $\binom{q}{t-1}$ new $K_t$'s from this star alone.

\begin{case}$q > c-d$
\end{case}

In this case, the $q$ new edges do not form a star in the colex component; they complete the next largest clique and then add a new vertex. The assumption that $\binom{q}{2} \le u$ implies that $q \le c$, so only one new vertex is added. Let $q_1 = c-d$ (the number of edges to complete the next clique) and $q_2 =q-q_1 = q-c+d$ (the degree of the new vertex). Note $q_2 \le d$. In completing the first clique we add $\binom{d}{t-2} + \binom{d+1}{t-2} + \cdots + \binom{c-1}{t-2}$ $K_t$'s. Then the edges incident to the new vertex add $\binom{0}{t-2} + \binom{1}{t-2} + \cdots + \binom{q_2-1}{t-2}$ $K_t$'s. The total number of new $K_t$'s is
\[\sum_{i=0}^{q_2-1} \binom{i}{t-2} + \sum_{i=d}^{c-1} \binom{i}{t-2} \ge \sum_{i=0}^{q_2-d+c-1} \binom{i}{t-2} = \sum_{i=0}^{q-1} \binom{i}{t-2} = \binom{q}{t-1}.\qedhere
\]
\end{proof}

\begin{theorem}\label{thm:colexfold}Suppose that $G \in \cG_C(m,r)$ and $T$ is a cluster in $G$ with $e(B) < e(R) \le r$. Let $G'_T$ be the colex folding of $G$ at $T$. Then $G'_T \in \cG(m,r)$ and, for every $t \ge 2$, $k_t(G'_T) \ge k_t(G)$.
\end{theorem}

\begin{proof}
First, $G'_T \in \cG(m,r)$ by definition since $e(B) < e(R)$. We will apply Lemma \ref{lem:numt} with $u=\binom{r+1}{2}-e(R)$ and $q=e(B)$. We have $e(B) < e(R) \le r$. Therefore
\[
	\binom{e(B)}{2} \le \binom{r}2 = \binom{r+1}{2}-r \le \binom{r+1}{2}-e(R),
\]
and so Lemma \ref{lem:numt} implies $k_t(\cC(\binom{r+1}{2}-e(R))) + \binom{e(B)}{t-1} \le k_t(\cC(\binom{r+1}{2}-e(R)+e(B)))$. Therefore
\begin{align*}
k_t(G) &= k_t(T\cup S) + k_t(B) + k_t(G\setminus(T\cup S))\\
	&\le k_t\parens[\Big]{\cC\parens[\Big]{\binom{r+1}{2}-e(R)}}+\binom{e(B)}{t-1} + k_t(G\setminus(T\cup S)) \text{ by Propositions \ref{prop:kkt} and \ref{prop:bluetbound}}\\
	&\le k_t\parens[\Big]{\cC\parens[\Big]{\binom{r+1}{2}-e(R)+e(B)}} + k_t(G\setminus(T\cup S)) \text{ by Lemma \ref{lem:numt}}\\
	&=k_t(G'_T).\qedhere
\end{align*}
\end{proof}

\begin{cor}\label{cor:colexfold}Suppose that $G \in \cG_C(m,r)$ and $T$ is a cluster in $G$ with $e(B) < e(R) \le r$. Let $G'_T$ be the colex folding of $G$ at $T$. Then $G'_T \in \cG(m,r)$ and $\k(G'_T) \ge \k(G)$.
\end{cor}

\begin{proof}
By Theorem \ref{thm:colexfold}, $\sum_{t\ge 2} k_t(G'_T) \ge \sum_{t\ge 2} k_t(G)$.
\end{proof}

\subsection{Partial Folding}\label{subsec:partial}

Partial folding is a simpler variant of folding, first used in \cite{KR19}, to consider when $e(B) < e(R)$. For a graph $G \in \cG_C(m,r)$ and $T\of V(G)$ a cluster, a \emph{partial folding of $G$ at $T$} is a graph obtained from $G$ by deleting all blue edges and adding any $e(B)$ of the red edges. The resulting graph is always in $\cG(m,r)$. The following lemma shows that if some cluster $T$ in $G$ has $e(B) < e(R)$ and $\abs{T} \ge \frac{r-1}{2}$, then $\k(G)$ would be increased by a partial folding of $G$ at $T$.

\begin{lemma}\label{lem:oneedge}
	Suppose that $G \in \cG_C(m,r)$ and $T$ is a tight clique in $G$ with $\abs{T} \ge \frac{r-1}{2}$ and $e(B) < e(R)$. Let $G'$ be a partial folding of $G$ at $T$. Then $G'\in \cG(m,r)$ and $\k(G')>\k(G)$.
\end{lemma}

\begin{proof}
Let $s = \abs{S_T}$. The loss from deleting the blue edges is at most $2^{s-2}e(B)$: for each blue edge $uv$ with $v \in S$, $v$ has $s-1-d_R(v)$ neighbors in $S$ and at most $d_R(v)-1$ neighbors in $V(G)\wo(T\cup S)$ other than $u$. In total, $u$ and $v$ have at most $s-2$ common neighbors, and any subset of them forms one clique containing $uv$.

The number of cliques gained from adding any one red edge is at least $2^{|T|}$, one for each subset of $T$. Therefore the net gain is at least $2^{|T|}e(R) - 2^{s-2}e(B) > (2^{|T|}-2^{s-2})e(R) \ge 0$, since $|T| \ge \frac{r-1}{2}$.
\end{proof}


\section{Averaging} 
\label{sec:averaging}

In this section we consider the connected graphs in which every cluster either is unfoldable or is small and has many edges missing from its neighborhood. This includes the graphs in which there are no clusters. We use an averaging argument to show that these graphs are not extremal.

\subsection{Fixed loss and unfoldable clusters}\label{subsec:FL}

In handling unfoldable clusters, we will make use of a concept introduced in \cite{CR14}, the \emph{fixed loss} of a cluster.

\begin{definition}[\cite{CR14}] The \emph{fixed loss} of a graph $R$ is
\[
	\fl(R) = \sum_{\substack{I\in \Ind(R)\\I\not=\0}} (2^{\d_I} - 1),
\]
where $\Ind(R)$ is the set of independent sets of $R$, and for any $I\of V(R)$, $\d_I = \min\setof{d_R(x)}{x\in I}$.
\end{definition}

The fixed loss $\fl(R_T)$ associated with a cluster $T$ is a crude upper bound on the number of cliques appearing in $G$ but not in $G_T$. We will use prior technical results about the fixed loss to show that unfoldable clusters have a large number of incident edges each contained in few $t$-cliques for every $t \ge 3$.

\begin{theorem}[{\cite[Theorem 4.4]{CR14}}]\label{thm:maxFL}
	If $R$ is a graph on $s$ vertices then
	\[
		\fl(R) \le \fl(K_s) = s(2^{s-1}-1).
	\]
\end{theorem}

The following lemma appears in \cite{CR14} in an incorrect form. The following version is correct. The corrected proof is available in \cite{CR14arxiv}.

\begin{lemma}\label{lem:sbound}
	If $T$ is a cluster in $G \in \cG_C(m,r)$ with $\k(G_T)\leq \k(G)$, then $\fl(R_T)\geq 2^r-2^{|T|}$ and $|T|\le \log_2(s)$, where $s = \abs{S_T}$.
\end{lemma}

The following theorem and proof were slightly modified from \cite{CR14}.

\begin{theorem}\label{thm:flell}
	Let $R$ be a graph on $s$ vertices having $\ell$ vertices of degree one.  Then
	\[
		\fl(R)\leq 5\cdot 2^{s-2}+(s-\ell-2)2^{s-\ell-1}.
	\]
\end{theorem}

\begin{proof}
	Let $L$ be the set of vertices of degree one.  We split up the sum computing $\fl(R)$ into two parts, the contributions of independent sets containing an element of $L$ and the rest.  To this end, let
	\begin{align*}
		\fl'(R)&=\sum_{\substack{I\in \Ind(R)\\I\cap L\neq \0}} 2^{\d_I}-1=\card\setof{I\in \Ind(R)}{I\cap L\neq \0},\quad\text{and}\\
		\fl''(R)&=\sum_{\substack{\0\neq I\in \Ind(R)\\I\cap L=\0}} 2^{\d_I}-1.
	\end{align*}
	To bound the first term, we consider two cases. If $R$ does not contain a $K_2$ component, then observe
	\[
		\fl'(R)=\card\setof{I\in \Ind(R)}{I\cap L\neq \0}\leq (2^{\ell}-1)2^{s-\ell-1}.
	\]
	This follows from the fact that no vertex of $L$ is adjacent to any other and therefore, given any nonempty subset $L'$ of $L$, at least one vertex of $R\wo L$ is excluded from $I$.  So there are at most $2^{s-\ell-1}$ independent sets contributing to $\fl'(R)$ of the form $L'\cup J$ where $L\cap J=\0$.
	
	Otherwise, $R$ contains a $K_2$ component $xy$. No independent set of $R$ contains both $x$ and $y$, and the independent sets containing $x$ are in bijection with those containing $y$. Therefore 
	\begin{align*}
		\fl'(R) &= \card\setof{I\in \Ind(R)}{x, y \notin I, I\cap L\neq \0} + 2\card\setof{I\in \Ind(R)}{x \in I, I\cap L\neq \0}\\
		&\le (2^{\ell-2}-1)(2^{s-\ell})+2(2^{\ell-2})(2^{s-\ell})\\
		&= (2^{\ell}+2^{\ell-1}-2)2^{s-\ell-1},
	\end{align*} where in the second line for each term we have counted the independent sets by their vertices in $L\wo \set{x,y}$ and then their vertices in $R \wo L$.
	
	On the other hand, writing $d_L(v)$ for $\abs{N(v)\cap L}$,
	\begin{align*}
		\fl''(R)&=\sum_{\substack{\0\neq I\in \Ind(R)\\I\cap L=\0}} \left(2^{\d_I}-1\right)\\
		&\leq \sum_{\substack{\0\neq I\in \Ind(R)\\I\cap L=\0}} \abs{I}(2^{\d_I}-1)\\
		&= \sum_{v\in V(R)\wo L}\; \sum_{\substack{v\in I\in \Ind(R)\\I\cap L=\0}} \left(2^{\d_I}-1\right)\\
	\intertext{}
		&\leq \sum_{v\in V(R)\wo L}\; \sum_{\substack{v\in I\in \Ind(R)\\I\cap L=\0}} \left(2^{d(v)}-1\right)\\
		&\leq \sum_{v\in V(R)\wo L} 2^{s-\ell-d(v)+d_L(v)-1}(2^{d(v)}-1)\\
		&= \sum_{v\in V(R)\wo L} \left(2^{s-\ell+d_L(v)-1}-2^{s-\ell-d(v)+d_L(v)-1}\right)\\
		&= 2^{s-\ell-1}\left(\sum_{v\in V(R)\wo L} 2^{d_L(v)}-\sum_{v\in V(R)\wo L}2^{d_L(v)-d(v)}\right)\\
		&\leq 2^{s-\ell-1}(2^{\ell}+s-\ell-1).
	\end{align*}
The fifth step above follows as in the proof of Theorem~\ref{thm:maxFL} and the final step uses the convexity of $2^x$ on the first term and ignores the second.

Combining these bounds, we have
\begin{align*}
	\fl(R)&=\fl'(R)+\fl''(R)\\
	&\leq 2^{s-\ell-1}(2^{\ell} + \max\set{0,2^{\ell-1}-1} + 2^{\ell}+s-\ell-2)\\
	&= \max\set{2^s + 2^{s-\ell-1}(s-\ell-2), 2^s +2^{s-2}+ 2^{s-\ell-1}(s-\ell-3)}\\
	&\le 5\cdot 2^{s-2} + 2^{s-\ell-1}(s-\ell-2).\qedhere
\end{align*}
\end{proof}

\begin{lemma}\label{lem:clusnum}
	Let $r\geq 3$ and $3 \le t \le r+1$.  If $T$ is a cluster in $G \in \cG_C(m,r)$ and $\k(G_T)\leq \k(G)$, then there are at least $2\binom{|T|}{2}$ edges $e$ incident to $T$ with $k_t(e) \le \binom{r-3}{t-2}$.
\end{lemma}

\begin{proof}
	We will show that there are at least $|T|-1$ vertices $v$ in $S$ with $d_R(v) \ge 2$. This is clear if $\abs{T} = 1$. Otherwise, we let $\ell$ be the number of vertices of $R$ of degree one.  If $\ell\geq s-|T|+2$, then by Theorem~\ref{thm:flell}, we would have $\fl(R)\leq 5\cdot 2^{s-2}+(|T|-4)2^{|T|-3}$. Note that this would imply
	\[
		2^r-2^{|T|}\leq \fl(R)\leq 5\cdot 2^{r-1-|T|}+(|T|-4)2^{|T|-3}\leq 5\cdot 2^{r-1-|T|}+\frac{1}8 s\log_2 s\leq 5\cdot 2^{r-1-|T|}+\frac{1}8 r\log_2 r,
	\]
	by Lemma~\ref{lem:sbound}. We are assuming that $\abs{T} \ge 2$, and no cluster can have size $r$ for then it would not be a maximal clique---by definition it would be contained in a $K_{r+1}$. By the convexity of $2^x$, since $2 \le |T| \le r-1$, we have $2^{|T|} + 2^{r+1-|T|} \le 2^2 +2^{r-1}$, and	
	\begin{align*}
		2^{r}-2^{|T|}&\leq 5\cdot 2^{r-1-|T|}+\frac{1}8 r\log_2 r\\
		2^{r-1}-4\le2^{r}-2^{|T|}-2^{r+1-|T|}&\le 2^{r-1-|T|}+\frac{1}8 r\log_2 r\le 2^{r-3} + \frac18 r\log_2 r\\
		3\cdot 2^{r-3}-4 &< \frac18 r\log_2r,
	\end{align*}
a contradiction for $r \ge 4$. When $r=3$, we also have $|T|=2$, contradicting $2^r - 2^{|T|} \le 5\cdot 2^{r-1-|T|}+(|T|-4)2^{|T|-3}$.

	Let $h$ be the number of vertices in $R$ of degree at least two.  Having shown that $\ell\leq s-|T|+1$, we know that $h\geq |T|-1$. Thus there are $|T|h \ge 2\binom{|T|}{2}$ edges $xy$ with $x\in T$, $y \in S$, and $d_R(y) \ge 2$, so $|N_G(x)\cap N_G(y)| \le r-3$. Therefore $k_t(e) = k_{t-2}(G[N(x)\cap N(y)]) \le \binom{r-3}{t-2}$.
\end{proof}

\subsection{Averaging calculations}

We will write $E_1$ for the set of tight edges and let $E_2 = E(G)\setminus E_1$. 

\begin{lemma}\label{lem:type1bound}
	If every cluster $T$ in $G \in \cG(m,r)$ has $|T| \le \frac{r}{2}$, then less than half the edges of $G$ are tight.
\end{lemma}

\begin{proof}Let $\cC$ be the set of clusters in $G$. By Lemma \ref{lem:clustructure}, each tight edge is in exactly one cluster, so $\abs{E_1} = \sum_{T \in \cC}\binom{\abs{T}}{2}$.

By Lemma \ref{lem:clustructure}, an edge is incident to at most two clusters. By counting in two ways the pairs $(e,T)$ where $e \in E_2$ is incident to the cluster $T$, $\abs{E_2} \ge \frac12 \sum_{T\in\cC}\abs{T}(r+1-\abs{T})$.

	\begin{align*}
	\frac{\abs{E_1}}{\abs{E_1}+\abs{E_2}} &\le \frac{\sum_{T \in \cC}\binom{\abs{T}}{2}}{\sum_{T \in \cC}\binom{\abs{T}}{2}+\frac12 \sum_{T\in\cC}\abs{T}(r+1-\abs{T})}\\
	&= \frac{\sum_{T \in \cC}\binom{\abs{T}}{2}}{\sum_{T \in \cC}\abs{T}r/2}\\
	&= \frac1r \cdot \frac{\sum_{T\in \cC} \abs{T}(\abs{T}-1)}{\sum_{T\in \cC} \abs{T}}\\
	&\le \frac1r \cdot \frac{\sum_{T\in \cC} \abs{T}(\frac{r}2-1)}{\sum_{T\in \cC} \abs{T}}\\
	&= \frac1r\parens[\Big]{\frac{r}2-1}\\
	&= \frac12-\frac1r < \frac12.\qedhere
	\end{align*}
\end{proof}

\begin{lemma}\label{lem:tsum}
If every cluster $T$ in $G \in \cG_C(m,r)$ either (1) is unfoldable, i.e, $\k(G) \ge \k(G_T)$, or (2) has $e(R) \ge r-2$, then for every $t \ge 4$, 
\[
	\sum_{e \in E(G)}k_t(e) \le \sum_{e \in E(G)}\binom{r-2}{t-2}.
\]
\end{lemma}

\begin{proof}For every edge $e$, we will define a modified version of $k_t(e)$. We let $u_e$ be $1$ if $e$ belongs to an unfoldable cluster and $0$ otherwise. We let $u'_e$ be the number of unfoldable clusters to which $e$ is incident if $k_t(e) \le \binom{r-3}{t-2}$ and $0$ otherwise. We set
\[
	k'_t(e) = k_t(e) + (\frac{u'_e}{2}-u_e)\binom{r-3}{t-3}.
\]

We show for all edges in $G$ that $k'_t(e) \le \binom{r-2}{t-2}$. Suppose first that $k_t(e) \le \binom{r-3}{t-2}$. Note $u_e \ge 0$ and, by Lemma \ref{lem:clustructure}, $u'_e \le 2$. For these edges we have
\[
	k'_t(e) \le \binom{r-3}{t-2}+ \frac22\binom{r-3}{t-3} = \binom{r-2}{t-2}.
\]
Now if $k_t(e) > \binom{r-3}{t-2}$, then we have 
\[
	k'_t(e) = k_t(e) - u_e\binom{r-3}{t-3}.
\]
If such an edge is in an unfoldable cluster ($u_e = 1$), then at least one edge is missing from its neighborhood (using the hypothesis that $G$ is connected), so 
\begin{align*}
	k'_t(e) = k_t(e) - \binom{r-3}{t-3} &= k_{t-2}(N(e)) - \binom{r-3}{t-3}\\
	&\le \binom{r-1}{t-2}-\binom{r-3}{t-4} - \binom{r-3}{t-3} = \binom{r-2}{t-2}.
\end{align*}
Finally, if $e$ has $k_t(e) > \binom{r-3}{t-2}$ and is not in an unfoldable cluster (so $u_e = u'_e = 0$ and $k'_t(e) = k_t(e)$), then either $e$ is not in a cluster, in which case $|N(e)| \le r-2$ and $k_t(e) \le \binom{r-2}{t-2}$, or $e$ is in a cluster of type (2), so there are at most $\binom{r-1}2 - (r-2) = \binom{r-2}2$ edges in $N(e)$. By Proposition~\ref{prop:kkt}, $k'_t(e) = k_t(e) = k_{t-2}(N(e)) \le \binom{r-2}{t-2}$.
Now we compute as follows. 
\begin{align*}
     \sum_{e \in E(G)} \binom{r-2}{t-2} \ge  \sum_{e \in E(G)} k'_t(e) 
         = \sum_{e\in E(G)} k_t(e) + \binom{r-3}{t-3} \sum_{e\in E(G)} \frac{u'_e}{2} - \binom{r-3}{t-3} \sum_{e\in E} u_e. 
\end{align*}
The final negative term is $\binom{r-3}{t-3}$ times the total number of edges in unfoldable clusters, i.e.,
\[
  \binom{r-3}{t-3} \sum_{\text{$T$ an unfoldable cluster}} \binom{\abs{T}}{2}.  
\]
On the other hand, by Lemma \ref{lem:clusnum}, the second term is at least 
\[
    \frac12\binom{r-3}{t-3} \sum_{\text{$T$ an unfoldable cluster}} 2\binom{|T|}{2}. 
\] 
Thus
\[
    \sum_{e \in E(G)} \binom{r-2}{t-2} \ge  \sum_{e\in E(G)} k_t(e). \qedhere
\]
\end{proof}

\begin{lemma}\label{lem:weakineq}
	Let $r \ge 3$ and $G \in \cG(m,r)$. If $|E_1| < |E_2|$, and, for every $t\ge 4$, we have $\sum_{e \in E(G)} k_t(e) \le \sum_{e \in E(G)} \binom{r-2}{t-2}$, then
	\[
		\sum_{t \ge 2}\frac{1}{\binom{t}{2}}\sum_{e \in E(G)}k_t(e) < g(m,r).
	\]
\end{lemma}

\begin{proof}We first split up the sum on the left side, use the given properties of $G$, and find an upper bound in terms of $m$ and $r$.
\begin{align*}
\sum_{t \ge 2}\frac{1}{\binom{t}{2}}\sum_{e \in E(G)}k_t(e) &= m + \frac13\sum_{e\in E(G)} k_3(e) + \sum_{t \ge 4}\frac{1}{\binom{t}{2}}\sum_{e \in E(G)} k_t(e)\\
&\le m +\frac13(\sum_{e\in E_1}(r-1) + \sum_{e\in E_2}(r-2)) + \sum_{t \ge 4}\frac{m}{\binom{t}{2}}\binom{r-2}{t-2}\\
&= m + \frac13(\sum_{e\in E_1}(r-1) + \sum_{e\in E_2}(r-2)) + \frac{m}{\binom{r}{2}}\sum_{t \ge 4}\binom{r}{t}\\
&< m + \frac13(\frac{m}2 (r-1) + \frac{m}2(r-2)) + \frac{m}{\binom{r}2}(2^r -\binom{r}{3}-\binom{r}{2}-r-1)\\
&= \frac m3(r-3/2) + \frac{m}{\binom{r}2}(2^r -\binom{r}{3}-r-1)\\
&= \frac{m}{\binom{r}2}(2^r -\binom{r}{3}-r-1 + \frac{r-3/2}{3}\binom{r}{2})\\
&= \frac{m}{\binom{r}2}(2^r + r^2/12-13r/12-1).\end{align*}

We will prove that this is less than the number of cliques in $aK_{r+1}\cup\cC(b)$; i.e., that
\[
	a(2^{r+1}-r-2)+\k(\cC(b)) - \frac{a\binom{r+1}2+b}{\binom{r}2}(2^r + r^2/12-13r/12-1) \ge 0.
\]
For fixed $r$ and $b$, the left side of this inequality is a linear function of $a$. The coefficient of $a$ depends only on $r$. For $r \ge 7$, we will show that the inequality holds when $a=1$. Then the special case $b=0$ implies that the coefficient of $a$ is positive and therefore that the inequality holds for $a \ge 1$.
\begin{align*}
&2^{r+1}-r-2+\k(\cC(b)) - \frac{\binom{r+1}2+b}{\binom{r}2}(2^r + r^2/12-13r/12-1)\\
&= 2^r \left(2 - \frac{r+1}{r-1} - \frac{b}{\binom{r}{2}}\right) -r-2 +\frac{\binom{r+1}2+b}{\binom{r}2}(r^2/12-13r/12-1)+\k(\cC(b))\\
&\ge 2^r \left(1 - \frac{2}{r-1} - \frac{\binom{c+1}{2}-1}{\binom{r}{2}}\right) -r-2 +\frac{\binom{r+1}2+\binom{c}{2}}{\binom{r}2}(r^2/12-13r/12-1)+2^c-c-1\\
\intertext{}
&\ge \begin{cases}
2^r (1 - \frac{2}{r-1} - \frac{\binom{r-2}{2}-1}{\binom{r}{2}}) -r-2 +\frac{\binom{r+1}2+\binom{1}{2}}{\binom{r}2}(r^2/12-13r/12-1) & 1 \le c \le r-3\\
-r-2 +\frac{\binom{r+1}2+\binom{r-2}{2}}{\binom{r}2}(r^2/12-13r/12-1)+ 2^{r-2}-(r-2)-1 & c = r-2\\
2^r (-2/r) -r-2 +\frac{\binom{r+1}2+\binom{r-1}{2}}{\binom{r}2}(r^2/12-13r/12-1)+2^{r-1}-r & c = r-1\\
2^r (-\frac{2(2r-1)}{r(r-1)}) -r-2 +\frac{\binom{r+1}2+\binom{r}{2}}{\binom{r}2}(r^2/12-13r/12-1)+2^{r}-r-1 & c = r
\end{cases}\\
&\ge 0 \text{ for } r \ge 7.
\end{align*}

For $3 \le r \le 6$ we will use the fact that $k_3(G) \le k_3(aK_{r+1}\cup\cC(b))$, which we proved in \cite{KR19} for $r \le 8$. It is enough to show $\sum_{t \ge 4}k_t(G) < \sum_{t \ge 4}k_t(aK_{r+1}\cup\cC(b))$, as it implies 
\begin{align*}
	\k(G) = m + k_3(G) + \sum_{t \ge 4}k_t(G) &< m + k_3(aK_{r+1}\cup\cC(b)) + \sum_{t \ge 4}k_t(aK_{r+1}\cup\cC(b))\\
	&= \k(aK_{r+1}\cup\cC(b)).
\end{align*}
By assumption, $\sum_{e \in E(G)} k_t(e) \le \sum_{e \in E(G)} \binom{r-2}{t-2}$ for every $t\ge 4$, so 
\begin{align*}
	\sum_{t \ge 4}k_t(G) &= \sum_{t \ge 4}\frac{1}{\binom{t}{2}}\sum_{e \in E(G)}k_t(e)\\
	& \le \sum_{t \ge 4}\frac{1}{\binom{t}{2}}\sum_{e \in E(G)}\binom{r-2}{t-2} = \frac{m}{\binom{r}{2}}\sum_{t \ge 4} \binom{r}{t} = \frac{m}{\binom{r}{2}}(2^r - \binom{r}3-\binom{r}2-r-1).
\end{align*}
We will use this inequality together with the fact that \[\sum_{t \ge 4}k_t(aK_{r+1}\cup\cC(b)) \ge \sum_{t \ge 4}k_t(aK_{r+1}) = a(2^{r+1}-\binom{r+1}3-\binom{r+1}2-(r+1)-1).\] Taking $a=1$,
\begin{align*} 
&\sum_{t \ge 4}k_t(K_{r+1}\cup\cC(b)) - \sum_{t \ge 4}k_t(G)\\
& \ge (2^{r+1}-\binom{r+1}3-\binom{r+1}2-(r+1)-1) - \frac{\binom{r+1}2+b}{\binom{r}{2}}(2^r - \binom{r}3-\binom{r}2-r-1)\\
&= \begin{cases}
1 & r=3\\
(26-b)/6 & r=4\\
(65-3b)/5 & r=5\\
2(249-11b)/15 & r=6
\end{cases}
\quad > 0, \text{
using the fact that } b \le \binom{r+1}{2}-1 \le 20.
\end{align*}
As in the $r\ge 7$ case, the inequality holds for $a \ge 1$ since it holds for $a=1$.
\end{proof}

\begin{theorem}\label{thm:everycluster}
If every cluster $T$ in $G \in \cG_C(m,r)$ either (1) is unfoldable, i.e., $\k(G) \ge \k(G_T)$, or (2) has $e(R) \ge r-2$ and $|T| \le r/2$, then $G$ is not extremal.
\end{theorem}

\begin{proof}
Lemma \ref{lem:sbound} states that the unfoldable clusters have $|T| \le \log_2(s)$, which implies the second inequality of
\[
	2|T|+1 \le |T| + 2^{|T|} \le r+1,  
\] so $|T| \le r/2$ for every cluster of $G$. Lemma \ref{lem:type1bound} shows that less than half the edges are tight. Lemma \ref{lem:tsum} shows that for every $t \ge 4$, $\sum_{e \in E(G)}k_t(e) \le \sum_{e \in E(G)}\binom{r-2}{t-2}$. Therefore the hypotheses of Lemma \ref{lem:weakineq} are satisfied, so $\k(G) = \sum_{t\ge 2}\frac{1}{\binom{t}{2}}\sum_{e\in E(G)}k_t(e) < g(m,r)$.
\end{proof}


\section{Proof of the Main Theorem} 
\label{sec:proof}

\begin{lemma}\label{lem:disco} Given $1 \le m_1 \le m_2$, let $G_i = a_iK_{r+1}\cup\cC(b_i)$, where $m_i = a_i\binom{r+1}2 +b_i$. Set $m = m_1 + m_2$ and let $G = aK_{r+1}\cup \cC(b)$, where $m = a\binom{r+1}2 +b$. Then $\k(G) \ge \k(G_1) + \k(G_2)$, with strict inequality unless $G_1 \cup G_2$ is an extremal graph according to the main theorem.
\end{lemma}

\begin{proof}
Reasonably straightforward with the help of Lemma 11 of \cite{KR19}.
\end{proof}

\begin{main}
	For all $m,r\in \N$, write $m=a\binom{r+1}{2}+b$ with $0\le b < \binom{r+1}2$. If $G$ is a graph on $m$ edges with $\Delta(G)\leq r$, then 
	\begin{equation}
		\k(G) \leq \k(aK_{r+1} \cup \cC(b)),
	\end{equation}
	with equality if and only if (disregarding isolated vertices) $G=aK_{r+1}\cup \cC(b)$ or $G=aK_{r+1}\cup K_c\cup K_2$ (where $b = \binom{c}{2}+1$).
\end{main}

\begin{proof}
	We disregard isolated vertices because they do not affect the number of edges, maximum degree, or number of cliques. For $r=1$, the theorem is trivial as $mK_2$ is the only possible graph. For $r=2$, it is almost as trivial, as we have $\k(G)\leq m+a$. The graphs $G \in \cG(m,2)$ that achieve equality have $k_3(G) = a$ so are $G=aK_{r+1}\cup \cC(b)$ and, when $b=2$, $G=aK_3\cup 2K_2$.
	
	For $r \ge 3$, we use induction on $m$. The Kruskal-Katona Theorem (Proposition \ref{prop:kkt}) implies the theorem for $m \le \binom{r+1}2$, i.e. $a=0$. By Lemma \ref{lem:disco}, if $G$ is disconnected, then by induction on a component $C$ and $G[V(G)\setminus C]$, we have $\k(G) \leq \k(aK_{r+1} \cup \cC(b))$, with equality if and only if $G=aK_{r+1}\cup \cC(b)$ (or, if $d=1$, $G=aK_{r+1}\cup K_c\cup K_2$, where $b = \binom{c}2+d$). Henceforth we assume $G$ is connected.
	
	If $G$ contains a cluster $T$ of any of the following types, then we use a local move to show $G$ is not extremal:
\begin{enumerate}
\item $\k(G) < \k(G_T)$ and $e(B) \ge e(R)$: then consider the folding $G_T$ of $G$ at $T$. Let $m' = e(G_T)$ (which is at most $m$). Then $G_T\cup (m-m')K_2 \in \cG(m,r)$ has $\k(G_T\cup (m-m')K_2) \ge \k(G_T) > \k(G)$, so $G$ is not extremal.
\item $e(B) < e(R) \le r$: then by Corollary~\ref{cor:colexfold}, the colex folding $G'_T$ of $G$ at $T$ has at least as many cliques as $G$ and is
 disconnected. If $\k(G) = \k(G'_T) = f(m,r)$, then $G'_T = aK_{r+1} \cup \cC(b)$ (or, if $d=1$, $G'_T=aK_{r+1}\cup K_c\cup K_2$). This is impossible because
  $G$ was connected, so $G'_T$ cannot contain a $K_{r+1}$. Therefore $G$ is not extremal.
\item $e(B) < e(R)$ and $|T| \ge \frac{r-1}{2}$: then a partial folding of $G$ at $T$ shows $G$ is not extremal, by Lemma \ref{lem:oneedge}.
\end{enumerate}

	Otherwise, each cluster of $G$ has none of the above types, so has either
\begin{enumerate}
\item $\k(G) \ge \k(G_T)$, or
\item $e(R) \ge r+1$ and $|T| \le \frac{r-2}{2}$.
\end{enumerate}
In this case, Theorem \ref{thm:everycluster} shows $G$ is not extremal.
\end{proof}


\section{Open Problems} 
\label{sec:conclusion}

There is a broad class of similar problems where one determines 
\[
	\mex_{T}(m,F) = \max\setof{n_T(G)}{\text{ $G$ is a graph with $m$ edges not containing a copy of $F$}},
\]
where $n_T(G)$ is the number of copies of $T$ in $G$. Conjecture \ref{conj:allt} concerns $\mex_{K_t}(m,K_{1,r+1})$. There are many natural problems of this form. They are the edge analogues of the problems introduced by Alon and Shikhelman in \cite{AlonS}. They define 
\[
	\ex_{T}(n,F) = \max\setof{n_T(G)}{\text{ $G$ is a graph on $n$ vertices not containing a copy of $F$}}
\]
and solve many problems of this form.

After the preparation of this paper, Chakraborti and Chen \cite{CC20} announced a proof of Conjecture \ref{conj:allt}.


\section*{Acknowledgments}

We thank Stijn Cambie for helpful discussions that improved the paper.


\providecommand{\bysame}{\leavevmode\hbox to3em{\hrulefill}\thinspace}
\providecommand{\MR}{\relax\ifhmode\unskip\space\fi MR }
\providecommand{\MRhref}[2]{%
\href{http://www.ams.org/mathscinet-getitem?mr=#1}{#2}
}
\providecommand{\href}[2]{#2}


\begin{thebibliography}{99}

\bibitem{AlonS}
N.~Alon and C.~Shikhelman, \emph{Many {$T$} copies in {$H$}-free graphs},
Journal of Combinatorial Theory, Series B \textbf{121} (2016), 146 -- 172,
Fifty years of The Journal of Combinatorial Theory.

\bibitem{BB}
B.~Bollobas, \emph{{Modern Graph Theory}}, Graduate Texts in Mathematics,
Springer New York, 1998.

\bibitem{CC20}
D.~{Chakraborti} and D.~Q. {Chen}, \emph{{Many cliques with few edges and
bounded maximum degree}}, arXiv e-prints (2020), \arxiv{2003.07943}.

\bibitem{C19}
Z.~{Chase}, \emph{The maximum number of triangles in a graph of given maximum
degree}, Advances in Combinatorics \textbf{2020:10} (2020), 5 pp.

\bibitem{CR14}
J.~Cutler and A.~J. Radcliffe, \emph{The maximum number of complete subgraphs
in a graph with given maximum degree}, Journal of Combinatorial Theory, Series B \textbf{104}
(2014), 60--71.

\bibitem{CR2016}
\bysame, \emph{The maximum number of complete subgraphs of fixed size in a
graph with given maximum degree}, Journal of Graph Theory \textbf{84} (2017),
no.~2, 134--145.

\bibitem{CR14arxiv}
\bysame, \emph{The maximum number of complete subgraphs in a graph with given
maximum degree}, arXiv e-prints (2021), \arxiv{1306.1803v2}.

\bibitem{EG14}
J.~Engbers and D.~Galvin, \emph{Counting independent sets of a fixed size in
graphs with a given minimum degree}, Journal of Graph Theory \textbf{76} (2014),
no.~2, 149--168.

\bibitem{FFK88}
P.~Frankl, Z.~F{\"u}redi, and G.~Kalai, \emph{Shadows of colored complexes},
Mathematica Scandinavica \textbf{63} (1988), no.~2, 169--178.

\bibitem{Frohmader08}
A.~Frohmader, \emph{Face vectors of flag complexes}, Israel Journal of Mathematics.
\textbf{164} (2008), 153--164.

\bibitem{G11}
D.~Galvin, \emph{Two problems on independent sets in graphs}, Discrete Mathematics
\textbf{311} (2011), no.~20, 2105--2112.

\bibitem{GLS15}
W.~Gan, P.~S. Loh, and B.~Sudakov, \emph{Maximizing the number of independent
sets of a fixed size}, Combinatorics, Probability and Computing \textbf{24} (2015), no.~3,
521--527.

\bibitem{K68}
G.~Katona, \emph{A theorem of finite sets}, Theory of Graphs (Proceedings of the Colloquium Held at Tihany, Hungary, September 1966), Academic Press, New York, 1968, pp.~187--207.

\bibitem{KR19}
R.~{Kirsch} and A.~J. {Radcliffe}, \emph{{Many triangles with few edges}},
Electronic Journal of Combinatorics \textbf{26} (2019), no.~2, P2.36.

\bibitem{K63}
J.~B. Kruskal, \emph{The number of simplices in a complex}, Mathematical
optimization techniques, University of California Press, Berkeley, California, 1963,
pp.~251--278.

\bibitem{RU18}
A.~J. {Radcliffe} and A.~{Uzzell}, \emph{{Stability and Erd\H{o}s--Stone type
results for $F$-free graphs with a fixed number of edges}}, arXiv e-prints
(2018), \arxiv{1810.04746}.

\bibitem{Z49}
A.~A. Zykov, \emph{On some properties of linear complexes}, Matemati\u{c}eski\u{\i} Sbornik N.S.
\textbf{24(66)} (1949), 163--188.

\end{thebibliography}
\end{document}